 \documentclass[12pt,a4paper]{amsart}

 \usepackage{amsmath, color}
 \usepackage{amsfonts}
 \usepackage{amssymb}
 \usepackage{hyperref}
\usepackage{enumerate}

\newtheorem{Th}{Theorem}[section]
\newtheorem{Lemma}[Th]{Lemma}
\newtheorem{Cor}[Th]{Corollary}

\newtheorem{Prop}[Th]{Proposition}

 \theoremstyle{definition}
\newtheorem{Def}[Th]{Definition}

\newtheorem{Ex}[Th]{Example}

\newcommand{\ZZ}{\ensuremath{\mathbb{Z}}}

\newcommand{\RR}{\ensuremath{\mathbb{R}}}

\newcommand{\Co}{\ensuremath{\mathbb{C}}}

\newcommand{\vt}{\mathbf{t}}
\newcommand{\va}{\mathbf{a}}
\newcommand{\vx}{\mathbf{x}}

\newcommand{\supp}{\mathrm{supp} \ }

\begin{document}

\title[Tiling and weak tiling]{Tiling and weak tiling in $(\ZZ_p)^d$}

\author[G. Kiss, D. Matolcsi, M. Matolcsi, G. Somlai]{Gergely Kiss, D\'avid Matolcsi, M\'at\'e Matolcsi, G\'abor Somlai}
\address{Gergely Kiss: Alfr\'ed R\'enyi Institute of Mathematics, ELKH, H-1364 Budapest, Hungary, and Eötvös Loránd University, Pázmány Péter Sétány 1/C, H-1111, Budapest, Hungary}
\email{kiss.gergely@renyi.hu}
\address{M\'at\'e Matolcsi: Department of Analysis, Institute of Mathematics, Budapest University of Technology and Economics (BME), M\"{u}egyetem rkp. 3., H-1111 Budapest, Hungary (also at Alfred Renyi Institute of Mathematics, ELKH, POB 127 H-1364, Budapest, Hungary)}
\email{matomate@renyi.hu}
\address{D\'avid Matolcsi: Eötvös Loránd University, Pázmány Péter Sétány 1/C, H-1111, Budapest, Hungary}
\email{matolcsidavid@gmail.com}
\address{Gábor Somlai:  Eötvös Loránd University, Pázmány Péter Sétány 1/C, H-1111, Budapest, Hungary, and Alfr\'ed R\'enyi Institute of Mathematics, H-1364 Budapest, Hungary.}
\email{gabor.somlai@ttk.elte.hu}

\thanks{G.K.\ was supported by NKFIH grants K124749  and K142993, by Bolyai János Research Fellowship of the Hungarian Academy of Sciences and by ÚNKP-22-5 New National Excellence Program of the Ministry for Culture and
Innovation.}
\thanks{M.M.\ was supported by NKFIH grants K129335 and K132097.}
\thanks{G.S.\ was supported by NKFIH grants 138596  and 132625, and G.S is a holder of János Bolyai Research Fellowship.
The work of G.S. on the project leading to this application has received funding from the European Research Council (ERC) under the European Union’s Horizon 2020 research and innovation programme (grant agreement No. 741420).}

\begin{abstract}
We discuss the relation of tiling, weak tiling and spectral sets in finite abelian groups. In particular, in elementary $p$-groups $(\ZZ_p)^d$, we introduce an averaging procedure that leads to a natural object of study: a 4-tuple of functions which can be regarded as a common generalization of tiles and spectral sets.  We characterize such 4-tuples for $d=1, 2$, and prove some partial results for $d=3$.
\end{abstract}

\maketitle

\bigskip

\section{Introduction}

\medskip

The concept of weak tiling was introduced recently in \cite{convex}, in connection with Fuglede's conjecture for convex bodies in $\RR^d$. In this note we will study the relation of tiling, weak tiling and spectral sets in finite abelian groups, with particular attention to elementary $p$-groups $(\ZZ_p)^d$. Our motivation is to give a new tool to prove the "spectral $\to$ tile" direction of Fuglede's conjecture in finite abelian groups.

\medskip

We begin by recalling the necessary notions and fixing our notation.

\medskip

The cardinality of any finite set $A$ will be denoted by $|A|$. We use the standard notation $A\pm B=\{a\pm b: a\in A, \ b\in B\}$, and $kA=\{ka: \ a\in A\}$, for any positive integer $k$. For any $n$, let $\ZZ_n$ denote the cyclic group of order $n$.

\medskip

Let $G$ be a finite abelian group. A {\it character} is a homomorphism from $G$ to the complex unit circle $\mathbb{T}$. The dual group, denoted by $\widehat{G}$, is the collection of all characters of $G$. For a function $f: G\to \Co$, the Fourier transform $\widehat{f}: \widehat{G}\to \Co$ is defined as
$$\widehat{f}(\gamma)=\sum_{x\in G} f(x)\gamma(x).$$

\medskip

A set $A\subset G$ is called {\it spectral} if the function space $L^2(A)$ admits an orthogonal basis consisting of characters restricted to $A$. Such an orthogonal basis of characters is called a {\it spectrum} of $A$. (The spectrum, if it exists, is not necessarily unique.)

\medskip

We say that a set $A\subset G$ {\it tiles} $G$ if there exists another set $B\subset G$ such that each $g\in G$ can be written uniquely as $g=a+b$, where $a\in A, b\in B$. We usually express this relation as $A\oplus B=G$ or, in the functional notation, $1_A\ast 1_B=1_G$, where $\ast$ denotes convolution. The convolution of any two functions $f, g: G\to \Co$ is defined in the standard way as $(f\ast g)  (x)=\sum_{y\in G} f(x-y)g(y)$.

\medskip

For any function $f:G\to \Co$ the function $f_{-}$ is defined as $f_{-}(x)=f(-x)$.
Some basic properties of the Fourier transform read as follows: $\widehat{f\ast g}=\widehat{f}\cdot \widehat{g}$, $\widehat{f_{-}}(\gamma )=\overline{\widehat{f}{(\gamma)}}$, $\widehat{f\ast f_{-}}=|\widehat{f}|^2$, and for real-valued even functions $\widehat{\widehat{f}}=|G|f$.

\medskip

Fuglede's conjecture \cite{fug} stated that a set $A\subset G$ is spectral if and only if it tiles $G$. The conjecture was formulated explicitly in $\RR^d$, but Fuglede already mentioned that the notions make sense in any locally compact abelian group $G$. In fact, the first counterexample by Tao \cite{tao} in $\RR^5$ was based on a counterexample in the finite group $(\ZZ_3)^5$. Since then, further counterexamples \cite{mfm, mk1} have been constructed (all based on examples in finite groups) to both directions of the conjecture in $\RR^d$, $d\ge 3$. The conjecture (both directions of it) remains open for $\mathbb{R}$ and $\mathbb{R}^2$. Some further connections between the discrete and the continuous settings have been revealed by Dutkay and Lai in \cite{dl}.

\medskip

In this paper we will restrict our attention to finite abelian groups, and mostly to the case $G=(\ZZ_p)^d$ with $p$ being a prime. For finite abelian groups, several results have been discovered in recent years \cite{fallon, fmv, ferguson, ios, ksrv1, ksrv2, ks1, klm, laba1, romkol, mattheus, rx1, rx2, zhang}.  For the purposes of this note, we single out the results of \cite{ios} and  \cite{ferguson, mattheus}. In \cite{ios} the authors prove both directions of Fuglede's conjecture in $(\ZZ_p)^2$. To the contrary, in \cite{ferguson} and \cite{mattheus} the authors prove (independently of each other) that for odd primes $p$ there exist spectral sets in $(\ZZ_p)^d$, $d\ge 4$, which do not tile the group, thus exhibiting a counterexample to the "spectral $\to$ tile" direction of Fuglede's conjecture in these groups.

\medskip

In connection with Fuglede's conjecture, the notion of {\it weak tiling} was introduced in \cite{convex}. For us, in the case of finite abelian groups, weak tiling can be formulated  as follows. A set $A\subset G$ tiles another set $E\subset G$ {\it weakly}, if there exists a nonnegative function $h: G\to \RR$ such that $1_A\ast h=1_E$. We will soon see that it makes sense to introduce the further restrictions that $h(0)=1$, the function $h$ is positive definite (i.e. $\widehat{h}\ge 0$), and restrict our attention to the case $E=G$.

\medskip

\begin{Def}\label{weaktile}
We say that a set $A\subset G$ {\it pd-tiles} $G$ {\it weakly} if there exists a nonnegative function $h: G\to \RR$ such that $h(0)=1$, $1_A\ast h=1_G$, and $\widehat{h}\ge 0$. \end{Def}

\medskip

In the terminology, "pd" stands for positive definite. Note that, due to positive definiteness, $h$ is necessarily an even function, $h(x)=h(-x)$. The assumption $h(0)=1$ is essential, otherwise the constant function $h=\frac{|G|}{|A|}$ would provide a weak pd-tiling for any set $A\subset G$. Note also, as a comparison with the terminology in \cite{convex},  that if $A$ pd-tiles $G$ weakly then $A$ tiles its complement weakly.

\medskip

A simple but important observation, also essentially contained in \cite{convex}, is that if $A$ pd-tiles $G$ weakly with a function $h$, then
\begin{equation}\label{supp}
A-A \ \cap \ \supp h=\{0\}.
\end{equation}

The reason is that $x=a-a'$ and $h(x)>0$ would imply $1_A\ast h(a)\ge h(0)1_A(a)+h(x)1_A(a')=1+h(x)>1$.

\medskip

Next, we show that a proper tiling always induces a weak pd-tiling of $G$.

\medskip

\begin{Lemma}\label{pdtile}
If $A\oplus B=G$ is a tiling, then $A$ pd-tiles $G$ weakly with $h_1=\frac{1}{|B|} 1_B\ast 1_{-B}$.
\end{Lemma}

\begin{proof}
$A\oplus B=G$ implies $\widehat{1}_A \cdot \widehat{1}_B = |G|\delta_0$, so that the supports of the nonnegative functions $\widehat{1}_A$ and $\widehat{1}_B$ are essentially disjoint (only intersect at 0).  This implies that $\widehat{1}_A \cdot \frac{1}{|B|}|\widehat{1}_B|^2 = |G|\delta_0$, which, in turn, implies $1_A\ast \left(\frac{1}{|B|}1_B\ast 1_{-B} \right)=1_G$.
\end{proof}

We remark that the notation $h_1$ above (instead of using simply $h$) is for later convenience.

\medskip

The essential observation in \cite{convex} is that if a set $A$ is spectral, then it tiles its complement weakly. However, slightly more is true.

\medskip

\begin{Lemma}\label{pdspectral}
Let $G$ be a finite abelian group. If $A\subset G$ is spectral, then $A$ pd-tiles $G$ weakly.
\end{Lemma}

\begin{proof}
Let $S\subset \widehat{G}$ be a spectrum of $A$. Then $|S|=|A|$ because the space $L^2(A)$ has dimension $|A|$, and $S$ is a basis. Let $h_1=\frac{1}{|A|^2}|\widehat{1}_S|^2$. Then $h_1\ge 0$,  $h_1(0)=1$, $\widehat{h}_1=\frac{|G|}{|A|^2}1_S\ast 1_{-S}\ge 0$, so $h_1$ is nonnegative, normalized and positive definite, as required. Also, the weak tiling condition  $1_A\ast h_1=1_G$ is most easily seen by taking Fourier transforms. $\widehat{1}_A \cdot \left(\frac{|G|}{|A|^2} 1_S\ast 1_{-S} \right)=|G|\delta(0)$ is true because $\widehat{1}_A(0)=|A|$, $1_S\ast 1_{-S}(0)=|S|=|A|$, and the support of $1_S\ast 1_{-S}$ is $S-S$, which is a subset of $\{\widehat{1}_A=0\}\cup \{0\}$ by the orthogonality of $S$ in $L^2(A)$.
\end{proof}

\medskip

By this lemma, we can establish the "spectral $\to$ tile" direction of Fuglede's conjecture in a finite abelian group $G$, if we prove that any set that pd-tiles $G$ weakly, actually tiles $G$ properly. This motivates the following definition.

\medskip

\begin{Def}\label{flat}
Assume that a finite abelian group $G$ has the property that whenever a set $A$ pd-tiles $G$ weakly then $A$ tiles $G$ properly. Then we call the group $G$ {\it pd-flat}. \end{Def}

\medskip

In the remainder of this note we will focus our attention on elementary $p$-groups $G=(\ZZ_p)^d$. As mentioned above, for odd primes and $d\ge 4$, there exist spectral sets in $G$ which do not tile $G$. Therefore, $(\ZZ_p)^d$ is not pd-flat for $d\ge 4$.

\medskip

In Section \ref{sec2} we introduce an averaging technique which leads to a natural generalization of spectral sets and tiles in $(\ZZ_p)^d$. Finally, in Section \ref{sec3} we show that $(\ZZ_p)^d$ is pd-flat for $d=1, 2$, and we give some partial results and conjectures for $d=3$.

\section{Averaging}\label{sec2}

It is natural to go one step further with the generalization of tiling.

\begin{Def} \label{functioweaktile}
Let $f, h: G\to \RR$ be nonnegative functions such that $f(0)=h(0)=1$, $\widehat{f}\ge 0, \widehat{h}\ge 0$. We say that the pair $(f, h)$ is a {\it functional pd-tiling} of $G$ if $f\ast h=1_G$. \end{Def}

\medskip

It turns out that this notion is very flexible, and for $G=(\ZZ_p)^d$ it gives rise to a natural averaging procedure. In connection with this, we introduce a special class of functions.

\medskip

\begin{Def}
We say that a function $f: (\ZZ_p)^d\to \Co$ is {\it ray-type}, if for any $\vx\in (\ZZ_p)^d$, and any $k=1, \dots, p-1$ we have $f(k\vx)=f(\vx)$. \end{Def}

\medskip

That is, a ray-type function is constant on any punctured line through the origin (but may have a different value at the origin).

\medskip

We also need some information on the zeroes of the Fourier transform of the indicator function $1_A$ of a set $A\subset (\ZZ_p)^d$. We make the identification $G=\{0, 1, \dots, p-1\}^d$,  and the same for the dual group $\widehat{G}=\{0, 1, \dots, p-1\}^d$. It is sometimes useful to think of the elements of $G$ as {\it column} vectors of length $d$, and elements of $\widehat{G}$ as {\it row} vectors of length $d$. Also, in notation, we will use boldface letters for elements of $G$ and $\widehat{G}$ to indicate that they are vectors. With these identifications in mind, the action of a character $\vt\in \widehat{G}$ on an element $
\vx\in G$ is given by $e^{2i\pi \langle \vt, \vx\rangle/p}$.

\medskip

For a function $f: G\to \Co$, and $\vt\in \widehat{G}$, we have $\widehat{f} (\vt)=\sum_{\vx\in G} f(\vx) e^{2i\pi \langle \vt, \vx\rangle /p}$. In particular, for $A\subset G$, the Fourier transform of $1_A$ takes the form $\widehat{1}_A(\vt)=\sum_{\va\in A} e^{2i\pi \langle \vt, \va \rangle /p}$.

\medskip

The important point here is that for any $\vt\ne 0$ we either have $\widehat{1}_A(k\vt)=0$ for all $k=1, 2, \dots, p-1$, or $\widehat{1}_A(k\vt)\ne 0$ for all $k=1, 2, \dots, p-1$. This is well-known, and the reason is that $\widehat{1}_A(\vt)=0$ if and only if the sum $\sum_{\va\in A} e^{2i\pi \langle \vt, \va \rangle /p}$ contains the same number of terms for each $p$th root of unity, in which case $\sum_{\va\in A} e^{2i\pi \langle k\vt, \va \rangle /p}$ also contains the same number of terms of each $p$th root of unity.

\medskip

Therefore, the zeros of the Fourier transform $\widehat{1}_A$ consist of punctured lines through the origin (for $\vt\ne 0$ a punctured line $\dot{L}$ is given by $\dot{L}=\{\vt, 2\vt, \dots, (p-1)\vt\}$; note that $0\notin \dot{L}$, and we use the notation $L=\dot{L}\cup \{0\}$). The same is true for the zeroes of the Fourier transform of $f=\frac{1}{|A|} 1_A\ast 1_{-A}$, because $\widehat{f}=\frac{1}{|A|}|\widehat{1}_A|^2$.

\medskip

We can now perform the first half of the averaging.

\begin{Lemma}\label{ave1}
Let $G=(\ZZ_p)^d$, and $1_A\ast h_1=1_G$ be a weak pd-tiling of $G$ by a set $A$. Then, for any $k=1, \dots, p-1$, $1_{kA}\ast h_1=1_G$ is also a weak pd-tiling. Furthermore, with the notation $f_k=\frac{1}{|A|}1_{kA}\ast 1_{-kA}$, we have that $f_k\ast h_1=1_G$ is a functional pd-tiling. Finally, for $f=\frac{1}{p-1}\sum_{k=1}^{p-1} f_k$ we have that $f$ is a ray-type function, $\supp f \ \cap \  \supp h_1=\{0\}$, and $f\ast h_1$ is also a functional pd-tiling.
\end{Lemma}
\begin{proof}
Note that $\widehat{1}_{kA}(\vt)=\widehat{1}_A(k\vt)$, so the zeroes of these functions are the same punctured lines through the origin.

\medskip

The assumption $1_A\ast h_1=1_G$ is equivalent to $\widehat{1}_A \widehat{h}_1=|G|\delta_0$. Due to the fact that the zeroes of $\widehat{1}_A$ and $\widehat{1}_{kA}$ coincide, and $|A|=|kA|$, we have $\widehat{1}_{kA} \widehat{h}_1=|G|\delta_0$, and therefore $1_{kA}\ast h_1=1_G$ is a weak pd-tiling.

\medskip

The function $f_k=\frac{1}{|A|}1_{kA}\ast 1_{-kA}$ is nonnegative, $f_k(0)=1$,  $\widehat{f}_k=\frac{1}{|A|}|\widehat{1}_{kA}|^2\ge 0$, and the zeroes of $\widehat{f}_k$ and $\widehat{1}_{kA}$ coincide. Also, $\widehat{f}_k(0)=|A|=\widehat{1}_{kA}(0)$, therefore $\widehat{f}_k \widehat{h}_1=|G|\delta_0$, which implies that $f_k\ast h_1$ is a functional pd-tiling. Also, $\supp f_k=kA-kA$, and hence $\supp f_k \ \cap \ \supp h_1=\{0\}$ by \eqref{supp}.

\medskip

Finally, for the average $f=\frac{1}{p-1}\sum_{k=1}^{p-1} f_k$, we clearly have $f\ge,0$, $\widehat{f}\ge 0$, $f(0)=1$ and  $f\ast h_1=|G|\delta_{0}$ is a functional pd-tiling since these properties hold for each $f_i$. Also, by averaging, it is clear that $f$ is ray-type, and the property $\supp f \ \cap \ \supp h_1=\{0\}$ is inherited.
\end{proof}

\medskip

We can perform a second step of averaging for the function $h_1$.

\begin{Lemma}\label{ave2}
Let $G=(\ZZ_p)^d$, and $1_A\ast h_1=1_G$ be a weak pd-tiling of $G$ by a set $A$, and define the ray-type function $f$ as in Lemma \ref{ave1}. For any $k=1, \dots, p-1$, let $h_k(\vx)=h_1(k\vx)$, and let $h=\frac{1}{p-1}\sum_{k=1}^{p-1}h_k$. Then, for each $k$, $f\ast h_k=1_G$ is a functional pd-tiling, the average $h$ is a ray-type function, and $f\ast h$ is also a functional pd-tiling such that $\supp f \ \cap \ \supp h=\{0\}$.
\end{Lemma}

\begin{proof}
Recall from Lemma \ref{ave1} that $f$ is a ray-type function. We claim that $\widehat{f}$ is also ray-type. To see this, note that $f$ can be written in the form $f=c\delta_0+\sum_{i} c_i1_{L_i}$ for some lines $L_i$ (for convenience, we use proper lines in this decomposition, not punctured lines; the difference is absorbed in the constant $c$ at the origin). The Fourier transform of $1_{L_i}$ is $p1_{H_i}$ with the hyperplane $H_i$ being orthogonal to $L_i$, and therefore the $\widehat{f}=c1_{\widehat{G}}+p\sum_i 1_{H_i}$, which is clearly ray-type.

\medskip

By Lemma \ref{ave1} we have that $f\ast h_1=1_{G}$, which implies $\widehat{f} \widehat{h}_1=|G|\delta_0$. This means that $\supp \widehat{f}\cap \supp \widehat{h}_1=\{0\}$. As $\widehat{f}$ is ray-type, its support is a union of lines $\dot{R}_i$, and hence $\widehat{h}_1$ must be 0 on the punctured lines $\dot{R_i}$. But $\widehat{h}_k(\vt)=\widehat{h}_1(\vt/k)$, so $\widehat{h}_k$ is also 0 on $\dot{R}_i$, and hence $\widehat{f} \widehat{h}_k=|G|\delta_0$ holds (for the value at zero we have $\widehat{h}_1(0)=\widehat{h}_k(0)=|G|/|A|$). In turn, this implies that $f\ast h_k=1_G$ is indeed a functional pd-tiling.

\medskip

Also, $\supp f \ \cap \ \supp h_1=\{0\}$ by Lemma \ref{ave1}, and $f$ is ray-type, so the support of $h_1$ is  contained in rays where $f=0$. This clearly implies that the support of $h_k$ is also contained in these rays, and therefore  $\supp f \ \cap \ \supp h_k=\{0\}$

\medskip

After averaging, it is clear that $h$ is ray-type, and the functional pd-tiling property $f\ast h$ is preserved, as well as the condition $\supp f \ \cap \ \supp h=\{0\}$.
\end{proof}

\medskip

We see from Lemma \ref{pdtile} and \ref{pdspectral} that $A$ pd-tiles $G$ weakly in both cases when $A$ tiles $G$, or $A$ is spectral in $G$. After applying Lemma \ref{ave1} and \ref{ave2} we arrive at the following corollary.

\begin{Cor}\label{magic}
Let $G=(\ZZ_p)^d$, and assume $A$ pd-tiles $G$ weakly with $1_A\ast h_1=1_G$ (in particular, this is the case if $A$
tiles $G$, or $A$ is spectral in $G$). For $k=1, \dots, p-1$, let $f_k=\frac{1}{|A|}1_{kA}\ast 1_{-kA}$, $h_k(\vx)=h_1(k\vx)$, and $f=\frac{1}{p-1}\sum_{k=1}^{p-1}f_k$, $h=\frac{1}{p-1}\sum_{k=1}^{p-1}h_k$.  Then
\begin{equation}\label{cara}
|A|=\sum_{x\in G}f(x),
\end{equation}
and the 4-tuple of functions $(f, h, \widehat{f}, \widehat{h})$ satisfy the following properties:
\begin{enumerate}[(i)]
\item  $f, h, \widehat{f}, \widehat{h}$ are all ray-type functions,
\item $f\ge 0$, $h\ge 0$,
\item\label{f0h0} $f(0)=1$, $h(0)=1$,
\item\label{fasth} $f\ast h=1_G$,
\item\label{sfh} $\supp f \cap \supp h=\{0\}$,
\item $\widehat{f}\ge 0$, $\widehat{h}\ge 0$,
\item $\widehat{f}(0)=|A|, \ \widehat{h}(0)=|G|/|A|$,
\item\label{fhhh} $\widehat {f}\ast \widehat{h}=|G|1_{\widehat{G}}$,
\item\label{sfhhh} $\supp \widehat{f} \cap \supp \widehat{h}=\{0\}$.
\end{enumerate}
\end{Cor}

\begin{proof}
Only \eqref{fhhh} needs further explanation, and it follows by taking Fourier transform of the equation, noting that $\widehat{\widehat{f}}=|G|f$ (and the same for $h$), and applying \eqref{sfh} and \eqref{f0h0}.
\end{proof}

\medskip

This corollary motivates the following definition.

\begin{Def}
A 4-tuple of functions $(f, h, \widehat{f}, \widehat{h})$ is called a {\it complementary 4-tuple} if they satisfy conditions (i)-(ix) of Corollary \ref{magic}. (We do not necessarily require that $f$ be constructed from a set $A$.)
\end{Def}

The notion of complementary 4-tuples is very appealing, because it puts the functions $f, h$ and their Fourier transform $\widehat{f}, \widehat{h}$  on equal footing. When studying spectral sets and tiles in $(\ZZ_p)^d$ it is natural to try to characterize all complementary 4-tuples. We will do this for $d=1, 2$ in the next section, and present some partial results for $d=3$.

\medskip

We also remark that an analogous averaging procedure can be carried out in cyclic groups $\ZZ_n$, leading to a similar notion of complementary 4-tuples in those groups. Studying these 4-tuples could lead to some insights concerning the Coven-Meyerowitz conjecture. This is subject to future research.

\section{Complementary 4-tuples in $(\ZZ_p)^d$}\label{sec3}

In this section we analyze complementary 4-tuples in  $(\ZZ_p)^d$, and prove that
$(\ZZ_p)^d$ is pd-flat for $d=1,2$, and also give some partial results for $d=3$.

\medskip

For convenience, we recall here what is known about  Fuglede's conjecture in these groups. Both directions of the conjecture are true for $d=1, 2$, with $d=1$ being fairly trivial, and $d=2$ being treated in \cite{ios}.
A simpler proof of the case $d=2$ can be found in \cite{ksrv1}.
For $d=3$ it is known that all tiles are spectral \cite{ios2}, but it is not known whether all spectral sets are tiles. For $d\ge 4$ (and $p$ being odd) there are examples of spectral sets which do not tile the group \cite{mattheus, ferguson} but the tile-spectral direction of the conjecture is open.

\medskip

We now turn to the main results of this section. We emphasize here that a group $G$ being pd-flat is {\it formally} stronger than the "spectral $\to$ tile" direction of Fuglede's conjecture in $G$.

\begin{Prop}\label{zp}
$G=(\ZZ_p)^d$ is pd-flat for $d=1, 2$, that is, every set $A$ which pd-tiles $G$ weakly, tiles $G$ properly. Also, $G$ is not pd-flat for $d\ge 4$.
\end{Prop}
\begin{proof}
Assume $A$ pd-tiles $G$ weakly. Then there exists a complementary 4-tuple $(f, h, \widehat{f}, \widehat{h})$ with the properties listed in Corollary \ref{magic}, and $f$ is constructed from $1_A$ as in Corollary \ref{magic}.

\medskip

For $d=1$ the support of any ray-type function is either $\{0\}$ or $G$. If $\supp f=\{0\}$ then $|A|=1$, and $A$ tiles $G$ trivially. If $\supp f=G$ then necessarily $\supp h=\{0\}$, and $h(0)=1$ implies $h=\delta_0$, and hence $f$ must be the constant function 1. Therefore, $|A|=p$ by \eqref{cara}, and hence $A=G$.

\medskip

For $d=2$ we must perform a case-by-case analysis of the function $f$. Let $c=\inf_{x\in G} f(x)\ge 0$, and let $L_i$ be the lines in the support of $f$. Then $f$ can be decomposed uniquely as $f=c1_G+\sum_i d_i1_{L_i}+ m\delta_0$ with some $d_i\ge 0$ and $m\in \RR$  (note that $L_i$ denotes the full line here, not the punctured line).

\medskip

If $c>0$ then $\supp f=G$, and hence $\supp h=\{0\}$, $h=\delta_0$, so $f=1_G$. By \eqref{cara}  we obtain that $|A|=p^2$ and $A=G$.

\medskip

If $c=0$ and $m>0$ then $\widehat{f}>0$ everywhere on $\widehat{G}$, so $\supp \widehat{h}=\{0\}$. Thus $h$ is constant, and $h(0)=1$ implies that $h$ is constant $1$. This implies that $f=\delta_0$, and $|A|=1$.

\medskip

If $c=0$ and $m=0$, then the total mass of $f$ is exactly $p$ times the mass at $0$, because this is true for all $1_{L_i}$. The mass at zero is $f(0)=1$, and hence we have $|A|=p$, by equation \eqref{cara}. Also, $c=0$ implies that $f$ must be zero on some punctured line $\dot{R}$, so $A$ can have at most one element in each coset of the line $R$ since $A-A \subset \supp(f)$. But $|A|=p$, so $A$ must have exactly one element in each coset of $R$, and therefore $A\oplus R =G$ is a tiling.

\medskip

Finally, we claim that $c=0$ and $m<0$ is not possible. Indeed, $c=0$ means that $f$ must be zero on some punctured line $\dot{R}$, and therefore $\widehat{f}=m<0$ on the orthogonal punctured line $\dot{R}^\perp$,  contradicting the nonnegativity of $\widehat{f}$.

\medskip

For $d\ge 4$ we know that there exist spectral sets in $(\ZZ_p)^d$ which do not tile the group \cite{mattheus, ferguson}. Therefore, $(\ZZ_p)^d$ cannot be pd-flat.
\end{proof}

\medskip

For $d=3$, we conjecture that $(\ZZ_p)^3$ is pd-flat, but unfortunately we cannot give a complete characterization of complementary 4-tuples, we can only prove some partial results in this direction. Therefore, the "spectral $\to$ tile" direction of Fuglede's conjecture remains open in this case. We have the following partial result.

\medskip

\begin{Prop}\label{de3}
Let $G=(\ZZ_p)^3$, and assume $A\subset G$ pd-tiles $G$ weakly with $1_A\ast h_1=1_G$. Let $(f, h, \widehat{f}, \widehat{h})$ be the complementary 4-tuple constructed in Corollary \ref{magic}.  Then either $A$ tiles $G$ properly, or the functions $f, h, \widehat{f}, \widehat{h}$ all have the following "dispersive" property: for any 2-dimensional subspaces $S\subset G$, $S'\subset \widehat{G}$ the intersections  $\supp f \cap S$, $\supp h \cap S$, $\supp \widehat{f}\cap S'$ and $\supp \widehat{h} \cap S'$ are all non-trivial (i.e. not equal to $\{0\}$ or the whole plane).
\end{Prop}

\begin{proof}
Exploiting the fact that all appearing functions are ray-type, we can represent each function in the following normalized form. Let us label the planes and lines through the origin as $S_1, S_2, \dots, S_{p^2+p+1}$ and $L_1, L_2, \dots, L_{p^2+p+1}$. Then $f$ can be represented uniquely as

\begin{equation}\label{decomp}
f=w 1_G+ \sum_i (c_i1_{S_i}+d_i1_{L_i})+m\delta_0,
\end{equation}
where the coefficients $w, c_i, d_i\ge 0$, $m\in \RR$ are defined by the following greedy algorithm: $w$ is the maximal value such that $f-w1_G\ge 0$,  $c_1$ is the maximal value such that $f- w1_G-c_1 1_{S_1}\ge 0$, $c_2$ is the maximal value such that $f-w1_G-c_11_{S_1}-c_21_{S_2}\ge 0$, etc, and subsequently $d_1$ is the maximal value such that $f-w1_G-(\sum_{i=1}^{p^2+p+1} c_i1_{S_i})- d_11_{L_1}\ge 0$, etc. We can represent $f, h, \widehat{f}, \widehat{h}$ in this manner uniquely, after fixing the order of planes and lines in $G$ and $\widehat{G}$.

\medskip

We will show that if any of $f, h, \widehat{f}, \widehat{h}$ does not have the dispersive property (stated in the proposition), then $A$ tiles $G$. This is done by a case-by-case analysis, where we will use the properties of complementary 4-tuples stated in Corollary \ref{magic}.

\medskip

Case I. Let us first consider the trivial case when the support of some of the appearing functions is the whole underlying group.

\medskip

(a) If $\supp f=G$, then $\supp h=\{0\}$ by \eqref{sfh}, and hence $h=\delta_0$ by \eqref{f0h0}, and $f=1_G$ by \eqref{fasth}, and $|A|=p^3$ by \eqref{cara}. Therefore, $A=G$, and $A$ tiles the group trivially.

\medskip

(b) Similarly, if $\supp h=G$ then $\supp f=\{0\}$, and $|A|=1$, and $A$ tiles $G$ trivially.

\medskip

(c) If $\supp \widehat{f}=\widehat{G}$, then $\supp \widehat{h}=\{0\}$ by \eqref{sfhhh}, and hence $h$ is a constant function, $h=1_G$ by \eqref{f0h0}, and therefore $f=\delta_0$ by \eqref{fasth}, and again $|A|=1$.

\medskip

(d) Similarly, if $\supp \widehat{h}=\widehat{G}$, then $\widehat{f}=\{0\}$, and hence $f$ is constant 1, and $|A|=p^3$ by \eqref{cara}.

\medskip

Having dealt with Case I, we can assume  for the rest of the argument that $w=0$ in equation \eqref{decomp}. Moreover, this also implies that $m\le 0$ in \eqref{decomp} because $m>0$ would imply $\supp \widehat{f}=\widehat{G}$. Altogether, we can conclude that the representation of  $f$ takes the form

\begin{equation}\label{decomp1}
f=\sum_i (c_i1_{S_i}+d_i1_{L_i})+m\delta_0,
\end{equation}
where $c_i, d_i\ge 0$, and $m\le 0$. The same is true for the representations of $h, \widehat{f}, \widehat{h}$.

\medskip

Case II. Next, assume that the support of one of the appearing functions intersects a plane only at 0.

\medskip

(a) Let $\supp f \ \cap \ S=\{0\}$. Let $\dot{L}\subset \widehat{G}$ denote the punctured line orthogonal to $S$. Then $\widehat{f}=m\le 0$ on $\dot{L}$, and the nonnegativity of $\widehat{f}$ implies $m=0$. Also, the "empty" plane $S$ intersects every other plane $S_i$, so every $S_i$ has at least one "empty" line $S\cap S_i$, and hence the weight $c_i$ must be 0 in equation \eqref{decomp1} for every $i$. Hence, $f$ can be represented in the form $f=\sum_i d_i 1_{L_i}$, and therefore the total mass of $f$ is $p$ times the mass at 0 (as this is the case for every $L_i$), and hence $|A|=p$ is implied by \eqref{cara} and $f(0)=1$. Furthermore, the fact that $S$ is empty means that $A-A \ \cap \ S=\{0\}$, so there is exactly one point of $A$ on each coset of $S$. Therefore, $A\oplus S=G$ is a tiling. 

\medskip

(b) If $\supp h \ \cap \ S=\{0\}$, then the same argument as in (a) shows that $h$ has a decomposition $h=\sum_i \tilde{d}_i L_i$, and therefore $\frac{\sum_{\vx\in G} h(\vx)}{h(0)}=p$. Considering $h(0)=1$, this implies $\sum_{\vx\in G} h(x)=p$, and hence $|A|=p^2$ by \eqref{fasth}. Also, $\supp f\ne G$, so there exists a line $L$ such that $A-A \ \cap \ L=\{0\}$, and hence $A\oplus L=G$ is a tiling.

\medskip

(c) If $\supp \widehat{f} \ \cap \ S'=\{0\}$, then the same argument as in (a) shows that $\widehat{f}$ has a decomposition $\widehat{f}=\sum_i d'_i L'_i$, and therefore $\frac{\sum_{\vt \in \widehat{G}} \widehat{f}(t)}{\widehat{f}(0)}=p$. Considering that $\widehat{f}(0)=\sum_{\vx\in G}f(x)$ and $\sum_{\vt \in \widehat{G}} \widehat{f}(t)=p^2 f(0)$, we obtain $\sum_{\vx\in G}f(\vx)=p^2$, that is $|A|=p^2$ by \eqref{cara}. We can finish the argument as in (b): as $\supp f\ne G$, there exists a line $L$ such that $A-A \ \cap \ L=\{0\}$, and hence $A\oplus L=G$ is a tiling.

\medskip

(d) If $\supp \widehat{h} \ \cap \ S'=\{0\}$, then the same argument as in (a) shows that $\widehat{h}$ has a decomposition $\widehat{h}=\sum_i \tilde{d}'_i L'_i$, and therefore $\frac{\sum_{\vt \in \widehat{G}} \widehat{h}(t)}{\widehat{h}(0)}=p$. This implies that
$\sum_{\vx\in G}h(\vx)=p^2$, and  $\sum_{\vx\in G}f(\vx)=p$. This implies that in the decomposition \eqref{decomp1} all coefficients $c_i=0$ and $m=0$, otherwise the total mass of $f$ would be greater than $p$ times the mass at 0 (recall that $c_i\ge$ and $m\le 0$). Therefore, $f=\sum_i d_i 1_{L_i}$. Also, as $\supp \widehat{f}\ne \widehat{G}$, we have a line $L'\subset \widehat{G}$ such that $\widehat{f}=0$ on $\dot{L'}$, which implies that $f=0$ on the punctured plane $S=L'^\perp$. Finally, as $|A|=p$ and $A-A \ \cap \ S=\{0\}$ we have that $A\oplus S=G$ is a tiling.

\medskip

Case III. Finally, assume that the support of one of the appearing functions contains a whole plane. In this case, we can invoke \eqref{sfh} and \eqref{sfhhh} to conclude that the support of some other function intersects the same plane at 0 only, and we are done by Case II above.
\end{proof}

It is worth summarizing here what this  result means for spectral sets. For a spectral set $A\subset (\ZZ_p)^3$, perform the avearaging procedure leading to the complementary 4-tuple in Corollary \ref{magic}. If any of the functions $f, h, \widehat{f}, \widehat{h}$ does not have the dispersive property described in Proposition  \ref{de3} then $A$ necessarily tiles $(\ZZ_p)^3$. If all appearing functions have the dispersive property, they must have a representation
\begin{equation}\label{disp}
\sum_i d_i 1_{L_i}+ m\delta_0,
\end{equation}
where the lines $L_i$ are situated such that they do not cover any plane fully, and $d_i> 0$, $m<0$ (the latter is true because $m\ge 0$ would imply the Fourier transform being positive on the planes $L_i^\perp$).

\medskip

It would be tempting to conjecture that this situation cannot occur, that is, complementary 4-tuples $(f, h, \widehat{f}, \widehat{h})$ with all functions having the dispersive property  do not exist. However, unfortunately, the following example shows that this is not the case.

\medskip

\begin{Ex}\label{David}
Let $S_1,S_2,S_3$ be three different planes through the origin in $(\ZZ_p)^3$, and let $L_1=S_2\cap S_3$, $L_2=S_1\cap S_3$, $L_3=S_1\cap S_2$ be the lines of intersection. For the sake of easier calculations we will use a representation for $f$ and $h$ different from \eqref{disp}. (It is not difficult to re-write the representations below according to equation \eqref{disp}, and we invite the reader to do so.)

\medskip

Let $f_0=2 \cdot 1_G- 2(1_{S_1}+1_{S_2}+1_{S_3})+p(1_{L_1}+1_{L_2}+1_{L_3})$, and $h_0=(1_{S_1}+1_{S_2}+1_{S_3})- 2(1_{L_1}+1_{L_2}+1_{L_3})+2p\delta_0$.

\medskip

Then,
$\widehat{f}_0= p^2(1_{L_1^\perp }+1_{L_2^\perp }+1_{L_3^\perp })-2p^2(1_{S_1^\perp}+1_{S_2^\perp}+1_{S_3^\perp})+2p^3\delta_0$,
and $\widehat{h}_0= 2p1_{\widehat{G}}-2p(1_{L_1^\perp }+1_{L_2^\perp }+1_{L_3^\perp })+p^2(1_{S_1^\perp}+1_{S_2^\perp}+1_{S_3^\perp})$.

\medskip

Based on these formulae, it is easy (but somewhat tedious) to check that the normalized functions $f=\frac{1}{3p-4}f_0$, $h=\frac{1}{2p-3}{h_0}$, and their Fourier transforms $\widehat{f}, \widehat{h}$ form a complementary 4-tuplet $(f, h, \widehat{f}, \widehat{h})$, such that none of the appearing functions have the dispersive property of Proposition \ref{de3}.

\medskip

It may be easier to visualize this example if we identify the punctured lines $\dot{L}_i$ of $(\ZZ_p)^3$ with points of the projective plane $PG(2,\mathbb{Z}_p)$. In this picture, $S_i$ become lines on the projective plane, and the functions $f$ and $h$ can be described in terms of a triangle in the projective plane. For example, the function $h$ is positive on the lines of a triangle (with the exception of the vertices, where $h$ is 0), and zero everywhere outside. We leave the details to the reader.
\end{Ex}

\medskip

Notice, however, that in this example neither $f$ nor $h$ can come from the averaging of an indicator function of a set $A$ as in Lemma \ref{ave1}. The reason is that by equation \eqref{cara} we would have $|A|=\frac{p^2(2p-3)}{3p-4}$ or $|A|=\frac{p(3p-4)}{2p-3}$, neither of which is an integer for $p\ne 2, 3$, and for $p=2, 3$ it is easy to check (with a finite case-by-case analysis) that $f$ and $h$ cannot be a result of averaging an indicator function of any set $A$.

\medskip

Example \ref{David} can be generalized: one can take $k$ points on the projective plane, with $k-1$ lying on a line, and one point being outside. Then, a complementary 4-tuple $(f, h, \widehat{f}, \widehat{h})$ can be constructed, such that the function $h$ is positive on all the lines connecting these points (with exception of the points themselves, where $h$ is 0), and zero everywhere outside. The function $f$ is then supported on the complement $\supp h$. The exact formulae are somewhat cumbersome, so we choose to omit them.

\medskip

At this point, we do not have a reasonable conjecture as to whether $(\ZZ_p)^3$ is pd-flat or not. Therefore, we cannot decide whether the "spectral $\to$ tile" direction of Fuglede's conjecture holds in these groups.

\medskip

As a final remark we mention here R\'edei's conjecture which concerns the structure of tilings in $(\ZZ_p)^3$. It states that in any normalized tiling $A\oplus B =(\ZZ_p)^3$ (normalized meaning that $0\in A, B$ is assumed) we must have that either $A$ or $B$ is contained in a proper subgroup.

\medskip

A complete characterization of complementary 4-tuples in $(\ZZ_p)^3$ could, in principle,  lead to the solution of both Fuglede's and R\'edei's conjecture in these groups.

\end{document}